\documentclass{tsp-short}
\usepackage{amsfonts}
\usepackage{amssymb}
\usepackage{newlfont}
\usepackage{amsthm}
\usepackage{euscript}
\usepackage{amsmath}
\usepackage{graphicx}
\usepackage{dsfont}
\usepackage{enumerate}
\usepackage{color}

\newcommand{\Var}{\mathop \mathrm{Var}}
\newcommand{\llim}{\mathop \mathrm{l.i.m.}}
\newcommand{\Plim}{\mathop \mathrm{\mathds{P}\!\!-\!\!lim}}

\newcommand{\Lip}{\mathop \mathrm{Lip}}

\newtheorem{Proposition}{Proposition}
\theoremstyle{definition}
\newtheorem{Example}{Example}
\newtheorem{Definition}{Definition}
\theoremstyle{plain}
\newtheorem{Lemma}{Lemma}
\theoremstyle{plain}
\newtheorem{Theorem}{Theorem}
\theoremstyle{condition}

\theoremstyle{plain}

\theoremstyle{remark}
\newtheorem{Remark}{Remark}

\newcommand{\veps}{\varepsilon}

\newcommand{\vp}{\varphi}
\newcommand{\vf}{\varphi}

\newcommand{\R}{\mathds{R}^d}
\renewcommand{\P}{\mathds{P}}
\newcommand{\E}{\mathds{E}}
\newcommand{\1}{1\!\!\!\,{\mathrm I}}

\begin{document}

\title[On differentiability of stochastic flow]{On differentiability of stochastic flow for a multidimensional SDE with discontinuous drift}

\author{Olga V. Aryasova}
\address{Institute of Geophysics, National Academy of Sciences of Ukraine,
Palladin pr. 32, 03680, Kiev-142, Ukraine}
\email{oaryasova@gmail.com}
\author{Andrey Yu. Pilipenko}
\address{Institute of Mathematics,  National Academy of Sciences of
Ukraine, Tereshchenkivska str. 3, 01601, Kiev, Ukraine; National Technical University of Ukraine "KPI", Kiev, Ukraine}
\email{pilipenko.ay@yandex.ua}

 \dedicatory{}


\begin{abstract}
We consider a $d$-dimensional SDE with an identity diffusion
matrix and a drift vector being a vector function of bounded
variation. We give a representation for the derivative of the solution with respect to the initial data.
\end{abstract}

\maketitle \thispagestyle{empty}
\section*{Introduction}
Consider an SDE of the form
\begin{equation}\label{eq_main}
\left\{
\begin{aligned}
d\vp_t(x)&=a(\vp_t(x))dt+dw_t,\\
\vp_0(x)&=x,\\
\end{aligned}\right.
\end{equation}
where $x\in\mathbb{R}^d, \ d\geq1, \ (w_t)_{t\geq0}$ is a $d$-dimensional Wiener process, $a=(a^1,\dots,a^d)$ is a bounded measurable mapping from $\R$ to $\R$.

According to \cite{Veretennikov81} there exists a unique strong solution to equation (\ref{eq_main}).

It is well known that if $a$ is  continuously differentiable and
its derivative is bounded, then equation (\ref{eq_main}) generates
a flow of diffeomorphisms. It turns out that this condition can be
essentially reduced \cite{Flandoli+10}, and a flow of
diffeomorphisms exists in the case of possible unbounded
H\"{o}lder continuous drift vector $a$. Recently the case of discontinuous drift was studied in \cite{Fedrizzi+13b, Fedrizzi+13a, MeyerBrandis+13, Mohammed+12} and the weak differentiability of the solution to (\ref{eq_main}) was proved  under rather weak assumptions on the drift. The authors of \cite{Fedrizzi+13b} consider a drift vector belonging to $L_q(0,T; L_p(\mathds{R}^d))$ for some $p,q $ such that
$$
p\geq2, \ q>2, \ \frac dp+ \frac 2q<1.
$$
They establish the existence of the G\^{a}teaux derivative in $L_2(\Omega\times[0,T];\R)$.
In \cite{Mohammed+12} it is proved that for a bounded measurable drift vector $a$ the solution belongs to the space $L^2(\Omega; W^{1,p}(U))$ for each $t\in\R, p>1,$ and any open and bounded $U\in\R$.  The Malliavin calculus is used in \cite{MeyerBrandis+13, Mohammed+12}.

The aim of our paper is to find a natural representation of the derivative $\nabla_x\vp_t(x)$ if $a$
is discontinuous.
We suppose that for $1\leq i \leq d$,  $a^i$ is a function of bounded
variation on $\R$, i.e. for each $1\leq j\leq d, $ the generalized
derivative $\mu^{ij}=\frac{\partial a^i}{\partial x_j}$ is a
signed measure on $\R$. Let $\mu^{ij,+},\mu^{ij,-}$ be measures from the Hahn-Jordan
decomposition $\mu^{ij}=\mu^{ij,+}-\mu^{ij,-}$. Denote
$|\mu^{ij}|=\mu^{ij,+}+\mu^{ij,-}$. Assume that for all
$1\leq i,j \leq d,$
$|\mu^{ij}|$ is a measure of Kato's class, i.e.
\begin{equation*}
\lim_{t\downarrow 0}\sup_{x\in\R}\int_{\R}\left(\int_0^t\frac{1}{(2\pi s)^{d/2}}\exp\left\{-\frac{\|y-x\|^2}{2s}\right\}ds\right)|\mu^{ij}|(dy)=0.
\end{equation*}
The condition we impose on the drift is more restrictive than that of \cite{Fedrizzi+13b, Mohammed+12}, but it allows us to obtain a representation for the derivative in terms of intrinsic parameters of the initial equation (see Theorem \ref{Theorem_main}). Our methods are different from those used in the  papers cited above.
We show that the derivative $Y_t(x)$ in $x$ is a solution of the following integral equation
$$
Y_t(x)=E+\int_0^tdA_s(\varphi(x))Y_s(x),
$$
where $A_t(\varphi(x))$ is a  continuous additive functional of
the process $(\varphi_t(x))_{t\geq0}$, which is equal to
$\int_0^t\nabla a(\varphi_s(x))ds$ if $a$ is differentiable, $E$
is the $d$-dimensional identity matrix.
This representation is a natural generalization of the expression for the derivative in the smooth case.

In the one-dimensional case (see \cite{Aryasova+12, Attanasio10}) the derivative was represented via the local time of the process. It is well known that the solution of (\ref{eq_main}) does not have a local time at a point in the multidimensional situation. We use continuous additive functionals for the representation of the derivative.  This method can be considered as a generalization of the local time approach to the multidimensional case.

Our method can be used in the case of non-constant diffusion and.

The paper is organized as follows. In Section
\ref{Section_W_functionals} we collect some definitions and
statements concerning continuous additive functionals. The main result of the paper
is formulated in Section \ref{Section_main} (see Theorem
\ref{Theorem_main}). For the proof we approximate equation
(\ref{eq_main}) by equations with smooth coefficients. The
definitions and properties of approximating equations are given in
Sections \ref{Section_Approximation},
\ref{Section_Convergence_Derivatives}. We prove Theorem
\ref{Theorem_main} in Section \ref{Section_Proof}.

\section{Preliminaries: W-functionals}\label{Section_W_functionals}
In this section we collect some facts about continuous additive functionals which will be
used in the sequel. Further information can be
found in  \cite{Blumental+07}; \cite{Dynkin63}, Ch. 6--8; \cite{Gikhman+04_II}, Ch.
II, \S 6.

Let $(\xi_t, \mathcal{M}_t, \P_x)$ be a homogeneous Markov process  with
a phase space $\R$ (see notations in \cite{Dynkin63}). Assume that $\xi_t, \ t\geq 0,$ has continuous trajectories and the infinite life-time.
Denote $\mathcal{N}_t=\sigma\{\xi_s:  0\leq s\leq t\}.$

\begin{Definition}\label{def_W_func}
A  random function $A_t, t\geq 0,$ adapted to the filtration $\{\mathcal{N}_t\}$   is called a  continuous additive functional of the process $(\xi_t)_{t\geq0}$ if it is
\begin{itemize}
\item  non-negative;
\item continuous in $t$;
\item homogeneous additive, i.e. for all $t\geq 0, \ s>0,$ $x\in\R,$
\begin{equation}\label{eq_additive}
A_{t+s}=A_s+\theta_sA_t \ \ \mathds{P}_x-\mbox{almost surely},
\end{equation}
where $\theta$ is the shift operator.
\end{itemize}
If additionally for each $t\geq 0,$
$$
\sup_{x\in\R}\mathds{E}_xA_t<\infty,
$$
then $A_t, \ t\geq 0,$ is called  a W-functional.
\end{Definition}

\begin{Remark}
It follows from Definition \ref{def_W_func} that a W-functional is non-decreasing as a function of  $t$, and for all $x\in\R$
$$
\mathds{P}_x\{A_0=0\}=1.
$$
\end{Remark}
\begin{Definition}
The function
$$
f_t(x)=\mathds{E}_x A_t, \ t\geq 0, \ x\in\R,
$$
is called the characteristic of a $W$-functional $A_t.$
\end{Definition}
\begin{Proposition}[See \cite{Dynkin63}, Theorem 6.3]\label{Proposition_uniquely_defined}
A W-functional is defined by its characteristic uniquely up to equivalence.
\end{Proposition}

The following theorem states the connection between the convergence of
functionals and the convergence of their characteristics.
\begin{Theorem}[See \cite{Dynkin63}, Theorem 6.4]\label{Theorem_Convergence_characteristics} Let $A_{n,t}, \
n\geq 1,$ be W-functionals of the process $(\xi_t)_{t\geq0}$  and $f_{n,t}(x)=\mathds{E}_x A_{n,t}$ be
their characteristics. Suppose that for each $t>0$, a function
$f_t(x)$ satisfies the condition
\begin{equation}\label{eq_uni_conv_char}
\lim_{n\to\infty}\sup_{0\leq u\leq t}\sup_{x\in\R}|f_{n,u}(x)-f_u(x)|=0.
\end{equation}
Then $f_t(x)$ is the characteristic of a W-functional $A_t$. Moreover,
\begin{equation*}
A_t=\llim_{n\to\infty} A_{n,t},
\end{equation*}
where $\llim$ denotes the convergence in mean square (for any initial distribution of $\xi_0$).
\end{Theorem}
\begin{Proposition}[See \cite{Dynkin63}, Lemma 6.1$^\prime$]\label{Proposition_uniform_convergence} If for any $t\geq 0$ the
sequence of non-negative additive functionals
$\left\{A_{n,t}: n\geq 1\right\}$ of the Markov process
$(\xi_t)_{t\geq 0}$ converges in probability to a continuous
functional $A_t$, then the convergence in probability is
uniform, i.e.
$$
\forall \ T>0 \sup_{t\in[0,T]}|A_{n,t}-A_t|\to 0, \
n\to\infty, \ \mbox{in probability}.
$$
\end{Proposition}
Let $h$ be a non-negative bounded measurable function on $\R$, let the process $(\xi_t)_{t\geq0}$  has a transition probability density $p_t(x,y)$. Then
$$
A_t:=\int_0^t h(\xi_s)ds
$$
is a $W$-functional of the process $(\xi_t)_{t\geq0}$ and its characteristic is equal to
$$
f_t(x)=\int_{\R}\left(\int_0^tp_s(x,y)ds\right)h(y)dy=\int_{\R}k_t(x,y)h(y)dy,
$$
where
$$
k_t(x,y)=\int_0^tp_s(x,y)ds.
$$

Let a measure $\nu$ be such that $\int_{\R}k_t(x,y)\nu(dy)$ is a function continuous in $(t,x).$ If we can choose a sequence of non-negative bounded continuous functions $\{h_n: n\geq1\}$ such that for each $T>0,$
$$
\lim_{n\to\infty}\sup_{t\in[0,T]}\sup_{x\in\R}\left|\int_{\R}k_t(x,y)h_n(y)dy-\int_{\R}k_t(x,y)\nu(dy)\right|=0,
$$
then by Theorem \ref{Theorem_Convergence_characteristics} there exists a W-functional corresponding to the measure $\nu$ with its characteristic being equal to $\int_{\R}k_t(x,y)\nu(dy).$ Formally we will denote this functional by $\int_0^t\frac{d\nu}{dy}(\xi_s)ds.$

A sufficient condition for the existence of a W-functional corresponding to a given measure is stated in the following theorem.
\begin{Theorem}[See \cite{Dynkin63}, Theorem 6.6]\label{Theorem_sufficient_condition}  Let the condition
\begin{equation}\label{Cond_A}
\lim_{t\downarrow 0}\sup_{x\in\R}f_t(x)=\lim_{t\downarrow 0}\sup_{x\in\R}\int_{\R}k_t(x,y)\nu(dy)=0
\end{equation}
hold. Then $f_t(x)$
is the characteristic of a W-functional
$A_t^{\nu}$. Moreover,
$$
A_t^{\nu}=\llim_{h\to 0}\int_0^t \frac{f_h(\xi_u)}{h}du,
$$
and the sequence of characteristics of functionals $\int_0^t \frac{f_h(\xi_u)}{h}du$ converges to $f_t(x)$ in sense of the relation \eqref{eq_uni_conv_char}.
\end{Theorem}

Let us return to the SDE (\ref{eq_main}). Let $(\varphi_t)_{t\geq0}$ be a solution of equation (\ref{eq_main}) with bounded measurable $a$.  The transition probability density $p_t^{\varphi}(y,z)$ of the process $(\varphi_t)_{t\geq0}$ satisfies the  Gaussian estimates (see \cite{Aronson67})
\begin{equation}\label{eq_gaussian_estimates}
\frac{K_1}{t^{d/2}}\exp\left\{-k_1\frac{\|y-z\|^2}{t}\right\}\leq p_t^{\varphi}(y,z)\leq \frac{K_2}{t^{d/2}}\exp\left\{-k_2\frac{\|y-z\|^2}{t}\right\}
\end{equation}
valid in every domain of the form $t\in[0,T], y\in\R, z\in\R,$ where $T>0$, $K_1, k_1, K_2, k_2$ are positive constants that depend only on $d, T,$ and $\|a\|_{\infty}.$

Denote by $k^w_t(x,y)$ the kernel $k_t(x,y)$ built on the transition density of the Wiener process, i.e.
\begin{equation}\label{eq_character_1}
k^w_t(x,y)=\int_0^t\frac{1}{(2\pi s)^{d/2}}\exp\left\{-\frac{\|y-x\|^2}{2s}\right\}ds.
\end{equation}
It is easily to see (\cite{Dynkin63}, Ch. 8, \S 1) that for all $x\in\R, \ y\in\R, \ x\neq y,$ $k_t^w(x,y)=\widetilde{k}_t(\|x-y\|),$ where
\begin{equation}\label{eq_character_2}
\widetilde {k}_t(r)=\frac{1}{2\pi^{d/2}}r^{2-d}\int_{r^2/2t}^{\infty}s^{d/2-2}e^{-s}ds, \ r>0.
\end{equation}
Therefore, the kernel $k^w_t(x,y)$ has a singularity if $x=y$ (for $d>1$) and the integral
\begin{equation*}\label{eq_character_finite}
f_t(x)=\int_{\R}k_t^w(x,y)\nu(dy)
\end{equation*}
is not well defined for all measures.

\begin{Definition}[see \cite{Kuwae+07}] A measure $\nu$ is a measure of Kato's class if
\begin{equation} \label{Cond_A_prime}
\lim_{t\downarrow 0}\sup_{x\in\R}\int_{\R}k_t^{w}(x,y)\nu(dy)=0.
\end{equation}
\end{Definition}

It follows from (\ref{eq_gaussian_estimates}) that a measure $\nu$ satisfies the condition (\ref{Cond_A}) if and only if it belongs to Kato's class.

\begin{Remark} \label{Cond_A_equiv}
A measure $\nu$ satisfies the condition (\ref{Cond_A_prime}) if and only if
\begin{eqnarray*}\label{eq_Kato_class}
\sup_{x\in\mathds{R}}\int_{|x-y|\leq1}\nu(dy)<\infty,& \mbox{when} \ d=1;\\
\lim_{\varepsilon\downarrow0}\sup_{x\in\mathds{R}^2}\int_{|x-y|\leq\varepsilon}\ln\frac{1}{|x-y|}\nu(dy)=0,& \mbox{when} \ d=2;\\
\lim_{\varepsilon\downarrow0}\sup_{x\in\mathds{R}^d}\int_{|x-y|\leq\varepsilon}|x-y|^{2-d}\nu(dy)=0,& \mbox{when} \ d\geq3.
\end{eqnarray*}
 The proof is a slight modification of that for the case of $\nu(dx)=f(x)dx$ given in \cite{Aizenman+82}, Theorem 4.5 (see  also \cite{Sznitman98}, Exercise 1 on p. 12). Here $f$ is a non-negative Borel measurable function. We use the representation (\ref{eq_character_2}) in the proof.
\end{Remark}

\begin{Example}\label{example_local_time} Let $d=1$. For each $y\in\R$, the measure $\nu=\delta_y$ belongs to Kato's class and corresponds to the W-functional
$$
L_t(y)=\lim_{\varepsilon\downarrow0}\frac{1}{2\varepsilon}\int_0^t \mathds{1}_{[y-\varepsilon,y+\varepsilon]}\left(w_s\right)ds,
$$
which is called the local time of a Wiener process at the point $y$.
  Assume that $\nu$ is a measure of Kato's class. This means now that $\sup_{x\in\mathds{R}}\nu([x,x+1])<\infty.$ Then (see \cite{Revuz+99}, Ch. X, \S 2)  the corresponding W-functional can be represented in the form
$$
A_t^{\nu}=\int_{\mathds{R}}L_t(y)\nu(dy).
$$
\end{Example}
\begin{Remark}
If $d\geq 2$, then $\delta_y$ does not belong to Kato's class. This is in consistency with the well-known fact that the local time at a point for a multidimensional Wiener process does not exist.
\end{Remark}

\begin{Example}\label{example_W_bound_function}
If $\nu(dy)=f(y)dy,$ where $f$ is a non-negative bounded function, then $\nu$
is a measure of Kato's class and $A^\nu_t=\int_0^tf(\xi_s)ds$.
\end{Example}
\begin{Example}\label{example_W_manifold}
Let $S\subset \R$ be a compact $(d-1)$-dimensional $C^1$-manifold. Denote by $\sigma_S$	 the surface measure.
Then for any non-negative bounded function $f$, the measure $\nu(dy)=f(y)\sigma_S(dy)$ belongs to Kato's class.
\end{Example}
\begin{Example}\label{example_W_Hausd_meas} Let $d\geq 2.$ Assume that a measure $\nu$ is such that
$$
\exists k, \gamma>0\ \forall x\in\R\ \forall \rho\in (0,1]:\ \ \ \ \nu(B(x,\rho))_\leq k\rho^{d-2+\gamma}.
$$
Then (c.f. \cite{Bass+03}, \S 2)
$$
\exists c=c(d,\gamma)\ \forall x\in \R\ \forall \rho\in(0,1]:\ \ \ \
 \int_{B(x,\rho)}|x-y|^{2-d}\nu(dy)\leq ck\rho^\gamma.
$$
This inequality together with Remark \ref{Cond_A_equiv}
yields
that $\nu$ is a measure of Kato's class. In particular, the
Hausdorff measure on the Sierpinski carpet in $\mathds{R}^2$ is such a measure  (see \cite{Bass+03}, Example 2.2).
\end{Example}

We will need the  uniform estimates on the moments of a W-functional.
\begin{Proposition}[\cite{Gikhman+04_II}, Ch. II, \S 6, Lemma 3]\label{Prop_Gikhman_Skorokhod} For all $m\geq 1, t>0,$
\begin{equation}\label{ineq_Gikhman_Skorokhod}
\sup_{x\in\R}\mathds{E}_x (A_t)^m\leq m!\left(\sup_{x\in\R}f_t(x)\right)^m.
\end{equation}
\end{Proposition}
Making use of this proposition one can easily obtain
the following modification of Khas'minskii's Lemma (see \cite{Khasminskii59} or \cite{Sznitman98}, Ch.1 Lemma 2.1).
\begin{Lemma}\label{Lemma_exponent_moment} Let the W-function $f_t$ satisfies the condition (\ref{Cond_A}). Let $A_t$ be the corresponding W-functional. Then for all $p>0,$ $t\geq0$, there exists a constant $C$ depending on $p, t,$ and $\|f_t\|_{\infty}$ such that for all $x\in\R,$
\begin{equation}\label{eq_exp_Moment_W-functional1}
\sup_{x\in\R}\mathds{E}_x \exp\left\{p A_t\right\}\leq C.
\end{equation}
\end{Lemma}

By the definition of a W-functional, $A_t^{\nu}$ is measurable w.r.t. the $\sigma$-algebra
generated by the Markov process. Since we have assumed that all the processes are continuous and have the infinite life-times, we may assume that $A_t^{\nu}=A_t^{\nu}(\cdot)$ is a measurable function defined on $C([0,\infty),\R)$ that depends only on behavior of functions on $[0,t]$ (if there  is no misunderstanding, sometimes we will consider $A_t^{\nu}$ as a function on $C([0,t],\R)$).

Let $(\vp_t(x))_{t\geq0}$ be a solution of \eqref{eq_main} defined on a probability space $(\Omega,\mathcal{F},\mathcal{F}_t,\P)$.
By $\mathds{P}_x$ denote the distribution of the process $(\varphi_t(x))_{t\geq0}$.
In Dynkin's notation \cite{Dynkin63} $((\vp_t(x))_{t\geq 0},\mathcal{F}_t,\P)$ is called a Markov family of random functions (the measures $\P_x$ are measures on the space of continuous functions, the  measure $\P$ is a probability on  $(\Omega,\mathcal{F})$).
The composition $A_t^{\nu}(\varphi_.(x)), t\geq 0,$ is an additive functional of $(\vp_t(x))_{t\geq 0}$ corresponding to  the measure $\nu.$ Note that $A_t^{\nu}(\varphi_.(x))$ is defined on $(\Omega,\mathcal{F},\mathcal{F}_t,\P)$
for any $x\in\R$.



If the measure $\nu$ belongs to Kato's class, then the corresponding  additive functionals of $(\varphi_t)_{t\geq0}$
and the Wiener process   are well defined.
Denote the corresponding measurable mappings by $A_t^{\nu,\varphi}$ and $A_t^{\nu,w}$.
By the Girsanov theorem, for each $x\in\R$,
the distributions  of  the processes
$(\varphi_t(x))_{t\geq0}$  and $(x+w_t)_{t\geq0}$ are equivalent. The question naturally arises whether the mappings $A_t^{\nu,\varphi}$ and $A_t^{\nu,w}$ are the same. The answer is positive and it is formulated in the next Lemma.
\begin{Lemma}\label{Lemma_equiv_w_functionals}
Let $\nu$ be a measure of Kato's class. Then for any $x\in\R$,
$$
A_t^{\nu,w}(\varphi_.(x))=A_t^{\nu,\varphi}(\varphi_.(x)) \ \ \P-\mbox{almost surely.}
$$
\end{Lemma}
\begin{proof}
For $x\in\R$, denote by $(w_t(x))_{t\geq0}$ the process $(x+w_t)_{t\geq0}$. According to Theorem \ref{Theorem_sufficient_condition},
$$
A_{t}^{\nu,w}(w_.(x))=\llim_{h\downarrow 0}\int_0^t\frac{f_h^w(w_s(x))}{h}ds.
$$
Then by the Girsanov theorem,
\begin{equation}\label{eq_Girsanov_theorem}
A_t^{\nu,w}(\varphi_.(x))=\Plim_{h\downarrow0}\int_0^t\frac{f_h^w(\varphi_s(x))}{h}ds,
\end{equation}
where $\Plim$ means the limit in probability.

It  remains to show that the characteristics of $\int_0^t\frac{f_h^w(\varphi_s(x))}{h}ds$
converge uniformly to $\int_{\R}k_t^{\varphi}(x,y)\nu(dy)$ as $h\downarrow 0$ (see Theorem \ref{Theorem_Convergence_characteristics}). This proof is routine and  technical, so we postpone it to Appendix.
\end{proof}

\section{The main result}\label{Section_main}
Let $a$ be a bounded measurable function of bounded variation.
Denote by
$\nabla a$ the matrix $\left(\frac{\partial a^i}{\partial
x_j}\right)_{1\leq i,j \leq d}$ and for $1\leq i,j\leq	d$, by $\mu^{ij}$ the signed measure $\frac{\partial a^i}{\partial x_j}$. Let $\mu^{ij}=\mu^{ij,+}-\mu^{ij,-}$ be the Hahn-Jordan decomposition of $\mu^{ij}$.
Further on we suppose that for all
$1\leq i,j \leq d,$ the measure
$|\mu^{ij}|=\mu^{ij,+}+\mu^{ij,-}$ belongs to Kato's class.

By Theorem  \ref{Theorem_sufficient_condition}, there exist $W$-functionals $A^{\mu^{ij, \pm},w}_t$ (we will denote the corresponding mappings by $A_t^{ij,\pm}(\cdot)$) with their characteristics defined according to the formula
$$
f_t^{ij,\pm}(x)=\int_{\R}k_t^w(x,y)\mu^{ij,\pm}(dy).
$$
Denote $A^{ij}_t=A^{ij,+}_t-A^{ij,-}_t, \ A_t=(A_t^{ij})_{1\leq i,j \leq d}.$

\begin{Remark}\label{Remark_HJ_dec} Assume that the measure $\mu^{ij}$ can be represented in the form $\mu^{ij}=\tilde{\mu}^{ij,+}-\tilde{\mu}^{ij,-}$, where $\tilde{\mu}^{ij,+},\tilde{\mu}^{ij,-}$ are from Kato's class and are not necessarily orthogonal. Then $A^{\mu^{ij,+}}-A^{\mu^{ij,-}}=A^{\tilde{\mu}^{ij,+}}-A^{\tilde{\mu}^{ij,-}}$.
\end{Remark}

\begin{Remark}
Recall that the mappings $A_t^{ij,+}, A_t^{ij,-}$ are continuous and monotonous in $t$. So the function $t\to A_t^{ij}$ is a continuous function of bounded variation  on $[0,T]$ almost surely.
\end{Remark}

The main result on differentiability of a flow generated by equation (\ref{eq_main}) with respect to the initial conditions is given in the following theorem.
\begin{Theorem}\label{Theorem_main}
Let $a:\R\to\R$ be  such that for all $1\leq i \leq d,$ $a^i$ is a function of bounded variation on $\R$. Put $\mu^{ij}=\frac{\partial a^i}{\partial x_j}, \ 1\leq i,j\leq d$.
Assume that the measures $|\mu^{ij}|, 1\leq i,j\leq d,$ belong to Kato's class.
Let $Y_t(x), \ t\geq 0$, be a solution to the integral equation
\begin{equation}\label{eq_derivative_main}
Y_t(x)=E+\int_0^t dA_s(\varphi(x))Y_s(x),
\end{equation}
where $E$ is the $d\times d$-identity matrix, the integral on the right-hand side of (\ref{eq_derivative_main}) is the Lebesgue-Stieltjes  integral with respect to the continuous function of bounded variation $t\rightarrow A_t(\varphi(x))$.

Then $Y_t(x)$ is the derivative of $\varphi_t(x)$ in $L_p$-sense: for all $p>0$,  $x\in\R$, $h\in\R$, $t\geq0$,
\begin{equation}\label{eq_derivative_main2}
\mathds{E}\left\|\frac{\varphi_t(x+\varepsilon h)-\varphi_t(x)}{\varepsilon}-Y_t(x)h\right\|^p\to 0, \ \varepsilon\to 0,
\end{equation}
where $\|\cdot\|$ is a norm in the space $\mathds{R}^d$.
Moreover,
$$
P\left\{\forall t\geq0: \varphi_t(\cdot)\in W_{p,loc}^1(\R,\R), \nabla\varphi_t(x)=Y_t(x) \ \mbox{for} \ \lambda\mbox{-a.a.} \ x\right\}=1,
$$
where $\lambda$ is the Lebesgue measure on $\R$.
\end{Theorem}
\begin{Remark}
The  differentiability was proved in \cite{Fedrizzi+13b, Mohammed+12}. We give a representation for the derivative. Note that the Sobolev derivative is defined up to the Lebesgue null set.
\end{Remark}

\begin{Remark}
Consider the non-homogeneous SDE
\begin{equation*}
\left\{
\begin{aligned}
d\vp_t(x)&=a(t,\vp_t(x))dt+dw_t,\\
\vp_0(x)&=x.\\
\end{aligned}\right.
\end{equation*}
Similarly to the arguments given in Section \ref{Section_W_functionals} a theory of non-homogeneous additive functionals of non-homogeneous Markov processes can be constructed. All the formulations and proofs can be literally rewritten with natural necessary modifications. Unfortunately, there are no corresponding references, therefore we did not carry out the corresponding reasonings.
\end{Remark}

Consider examples of functions $a$ for which $|\mu^{ij}|, 1\leq i,j\leq d,$ are measures of Kato's class.
\begin{Example}\label{Remark_Cond_For_Smooth_Functions}
Let for all $1\leq i\leq d$, ${a^i}$ be a Lipschitz function. By Rademacher's theorem \cite{Federer69} the Frech\'et derivatives $\mu^{ij}=\frac{\partial a^i}{\partial x_j}$ exist almost surely w.r.t. the Lebesgue measure. It is easy to verify that they are bounded and the Frech\'et derivative coincides with the derivative considered in the generalized sense. Then $|\mu^{ij}|$ belongs to Kato's class.

Let now $h\in C^1(\R, \R), D$ be a bounded domain in $\R$ with $C^1$ boundary $\partial D$. Put $a(x)=h(x)\mathds{1}_{x\in D}$. It follows from Example \ref{example_W_manifold} that for all $1\leq i,j\leq d,$ $|\mu^{ij}|$ is a measure of Kato's class because (cf. \cite{Vladimirov67})
$$
\mu^{ij}(dx)=\frac{\partial a^i}{\partial x_j}(x)\mathds{1}_{x\in{D}}dx+ h^i(x)\cos (n_j(x))\sigma_{\partial D}(dx),
$$
where $n(x)=(n_1(x),\dots,n_d(x))$ is the outward unit normal vector at the point $x\in\partial D.$

Condition (\ref{Cond_A_prime}) is also satisfied by the measure generated by $a$ being a linear combination of the form
\begin{equation}\label{eq_Func_a}
h_0(x)+\sum^m_{k=1}h_k(x)\mathds{1}_{x\in D_k},
\end{equation}
where $h_0\in\Lip(\R,\R)$, $h_k\in C^1(\R,\R), \ 1\leq k\leq d,$ $D_k$ is a bounded domain in $\R$ with $C^1$ boundary.

Further examples of $a$ can be obtained as the limits of sequences of the functions of form (\ref{eq_Func_a}).

In one-dimensional case all the functions of bounded variation generate measures belonginig to Kato's class (see Example \ref{example_local_time}).

See also Example \ref{example_W_Hausd_meas} showing that if $|\mu^{ij}|$ are
``Hausdorff-type'' measures with a parameter greater than $(d-1)$, then $a$ satisfies assumptions of the Theorem.
\end{Example}

The idea of the proof of Theorem \ref{Theorem_main} is to
approximate the solution of equation (\ref{eq_main}) by solutions
of SDEs with smooth coefficients.
The definition and properties of approximating equations are given
in Sections \ref{Section_Approximation},
\ref{Section_Convergence_Derivatives}. The proof of the Theorem itself
is presented in Section \ref{Section_Proof}.

\section{Approximation by SDEs with smooth
coefficients}\label{Section_Approximation} For $n\geq1,$ let
$g_n\in C_0^{\infty}(\R)$ be a non-negative function such that
$\int_{\mathds{R}^d}g_n(z)dz=1$, and $g_n(x)=0, \ |x|\geq 1/n$.
Put
\begin{equation}\label{eq_a_n}
a_n(x)=(g_n\ast a)(x)=\int_{\mathds{R}^d} g_n(x-y)a(y)dy, \
x\in\mathds{R}^d, \ \ n\geq1,
\end{equation}
where the function $a$ satisfies the assumptions of Theorem \ref{Theorem_main}.
Note that
\begin{equation}\label{eq_ineq_for_norms}
\sup_n\|a_n\|_{\infty}\leq \|a\|_{\infty},
\end{equation}
and $a_n\to a, \ n\to \infty,$ in  $L_{1,loc}(\R).$ Passing to
subsequences we may assume without loss of generality that
$a_n(x)\to a(x), \ n\to\infty,$ for almost all $x$ w.r.t. the
Lebesgue measure.

Consider the SDE
\begin{equation}\label{eq_main_n} \left\{
\begin{aligned}
d\vp_{n,t}(x)&=a_n(\vp_{n,t}(x))dt+dw_t,\\
\vp_{n,0}(x)&=x, \ x\in\mathds{R}^d.
\end{aligned}\right.
\end{equation}
Put $\nabla a_n=\left(\frac{\partial a_n^i}{\partial
x_j}\right)_{1\leq i,j \leq d}$. Denote by $Y_{n,t}(x)$ the
matrix of derivatives of $\varphi_{n,t}(x)$ in $x$, i.e.,
$Y_{n,t}^{ij}(x)=\frac{\partial\varphi_{n,t}^i(x)}{\partial x_j}.$ Then $Y_{n,t}(x)$ satisfies the equation
\begin{equation}\label{eq_derivative}
Y_{n,t}(x)=E+\int_0^t\nabla a_n(\vp_{n,s}(x))Y_{n,s}(x)ds,
\end{equation}
where $E$ is the $d$-dimensional identity matrix.
\begin{Lemma}\label{Lemma_Converg_Solutions} {\it For each $p\geq1$,
\begin{enumerate}[1)]
\item for all $t\geq0$ and any compact set $U\in\mathds{R}^d$,
$$
\sup_{x\in U, \ n\geq
1}\left(\mathds{E}(\|\vp_t^n(x)\|^p+\|\vp_t(x)\|^p)\right)<\infty;
$$
\item for all   $x\in\mathds{R}^d, \ T\geq0,$
$$
\mathds{E}\left(\sup_{0\leq t\leq
T}\|\vp^n_t(x)-\vp_t(x)\|^p\right)\to 0 \ \mbox{as} \ n\to \infty,
$$
where $\|\cdot\|$ is a norm in the space $\R$.
\end{enumerate} }
\end{Lemma}
\begin{proof}
Statement 1) follows from the uniform boundedness of the coefficients
and the finiteness of the moments of a Wiener process; 2) is proved in
\cite{Luo11}, Theorem 3.4.
\end{proof}

For  $1\leq i,j \leq d,$ put $\mu_n^{ij}=\frac{\partial
a_n^i}{\partial x_j}.$   By the properties of convolution of a
generalized function (see \cite{Vladimirov67}, Ch. 2, \S7),
$$
\nabla a_n=\nabla a\ast g_n.
$$
For each $n\geq
1$, $1\leq i,j \leq d,$ put $\mu_n^{ij,\pm}=\mu^{ij,\pm}\ast g_n$ and $\mu_n^{ij}=\mu_n^{ij,+}-\mu_n^{ij,-}$ (c.f. Remark \ref{Remark_HJ_dec}).  Then, according to
Theorem \ref{Theorem_sufficient_condition},
there exist W-functionals $A_{n,t}^{ij,\pm}$ of a Wiener process on $\R$ which
correspond to the measures $\mu_n^{ij,\pm}$ and have
characteristics of the form
\begin{equation}\label{eq_f_t}
f_{n,t}^{ij,\pm}(x)=\int_{\R}k_t^w(x,y)\mu_n^{ij,\pm}(dy), \ 1\leq
i,j\leq d.
\end{equation}
The
functional $A_{n,t}^{ij}=A_{n,t}^{ij,+}-A_{n,t}^{ij,-}$
is given by the formula
\begin{equation}\label{eq_functional_n}
{A_{n,t}^{ij}}=\int_0^t  \frac{\partial a_n^i}{\partial x_j}(w_u)du
\end{equation}
(see Example 2).

\begin{Lemma}\label{Lemma_Converg_w_Functionals}

{\it For each} $T>0$, $x\in\R$, $\varepsilon>0$, $1\leq i,j \leq d$,
$$
\P_{w(x)}\left\{\sup_{0\leq t\leq T}\left|A_{n,t}^{ij,\pm}-A_t^{ij,\pm}\right|>\varepsilon\right\}\to 0, \ n\to\infty,
$$
where $\P_{w(x)}$ is the distribution of the process $(x+w_t)_{t\geq0}.$

\end{Lemma}
The following simple proposition used for the proof of Lemma
\ref{Lemma_Converg_w_Functionals} is easily checked.
\begin{Proposition}\label{Proposition_convolution_characteristics}
Let $\nu, \nu_n, \ n\geq 1,$ be from the Kato class, $f, f_n, \ n\geq 1,$ be the characteristics of the corresponding W-functionals of a Wiener process, and  the representation
$\nu_n=g_n\ast\nu$ hold true. Then the relation $f_{n,t}=g_n\ast
f_t$ is fulfilled.
\end{Proposition}
\begin{proof}[Proof of Lemma \ref{Lemma_Converg_w_Functionals}]
To prove the convergence of functionals in mean square it is sufficient to show that for each $T>0$, $1\leq i,j\leq
d,$
\begin{equation}\label{eq_converg_character}
\lim_{n\to\infty}\sup_{0\leq t \leq
T}\sup_{x\in\R}|f_{n,t}^{ij,\pm}(x)-f_{t}^{ij,\pm}(x)|=0
\end{equation}
 (see Theorem
\ref{Theorem_Convergence_characteristics}). Then the uniform convergence in probability follows from Proposition \ref{Proposition_uniform_convergence}.

For each $0<\delta<t,$
\begin{multline*}
\sup_{x\in\R}\left|f_{n,t}^{ij,\pm}(x)-f_t^{ij,\pm}(x)\right|=
\sup_{x\in\R}\left|\int_{\R}k_t^w(x,y)\left(\mu_n^{ij,\pm}(dy)-\mu^{ij,\pm}(dy)\right)\right|=
I^{\pm}+II^{\pm},
\end{multline*}
where
\begin{equation}\label{eq_I}
I^{\pm}=\sup_{x\in\R}\left|\int_{\R}\left(\mu_n^{ij,\pm}(dy)-\mu^{ij,\pm}(dy)\right)\int_0^{\delta}\frac{1}{(2\pi
s)^{d/2}}\exp\left\{-\frac{\|y-x\|^2}{2s}\right\}ds\right|,
\end{equation}
$$
II^{\pm}=\sup_{x\in\R}\left|\int_{\R}\left(\mu_n^{ij,\pm}(dy)-\mu^{ij,\pm}(dy)\right)\int_{\delta}^t\frac{1}{(2\pi
s)^{d/2}}\exp\left\{-\frac{\|y-x\|^2}{2s}\right\}ds\right|.
$$
We have
\begin{multline*} I^{\pm}\leq
\sup_{x\in\R}\int_{\R}|\mu_n^{ij}|(dy)\int_0^{\delta}\frac{1}{(2\pi
s)^{d/2}}\exp\left\{-\frac{\|y-x\|^2}{2s}\right\}ds+\\
\sup_{x\in\R}\int_{\R}|\mu^{ij}|(dy)\int_0^{\delta}\frac{1}{(2\pi
s)^{d/2}}\exp\left\{-\frac{\|y-x\|^2}{2s}\right\}ds=I_1+I_2.
\end{multline*}
Because of the condition (\ref{Cond_A_prime}), for each $\varepsilon>0$, we
can choose $\delta$ so small that $I_2$ is less then
$\varepsilon/4$. To obtain the same estimate for $I_1$, note that
by the associative, distributive and commutative properties of
convolution (see \cite{Vladimirov67}, Ch. II, \S 7),
\begin{multline*}
I_1=\sup_{x\in\R}(|\mu_n|\ast k_\delta)(x)\leq\left((|\mu|\ast
g_n)\ast k_{\delta}\right)(x)=\sup_{x\in\R}\left(|\mu|\ast(g_n\ast
k_{\delta})\right)(x)=\\
\sup_{x\in\R}\left(|\mu|\ast(k_\delta\ast g_n)\right)(x)=
\sup_{x\in\R}\left((|\mu|\ast k_{\delta})\ast g_n\right)(x)
\leq \sup_{x\in\R}\left(|\mu|\ast k_{\delta}\right)(x)=I_2<\varepsilon/4.
\end{multline*}
We get $I^{\pm}<\varepsilon/2.$

Consider $II^{\pm}$. The functions
$$
q^{ij,\pm}_{\delta,t}(x):=\int_{\R}\mu^{ij,\pm}(dy)\int_\delta^t\frac{1}{(2\pi s)^{d/2}}\exp\left\{-\frac{\|x-y\|^2}{2s}\right\}ds
$$
are  equicontinuous in $x$ for $t\in[\delta,T]$. We have
\begin{equation*}
\sup_{\delta<t<T}II^{\pm}=\sup_{\delta<t<T}\sup_{x\in\R}|(q_{\delta,t}^{ij,\pm}\ast g_n)(x)-q_{\delta,t}^{ij,\pm}(x)|\to 0, \ n\to\infty.
\end{equation*}
Then there exists $n_0$ such that for all $n> n_0$, $\sup_{\delta<t<T}II^{\pm}<\varepsilon/2$.
\end{proof}
\begin{Lemma}\label{Lemma_Converg_Functionals}
{\it For each} $ T>0,\ x\in\R$, $\varepsilon>0$, $1\leq i,j \leq d$,
$$
\mathds{P}\left\{\sup_{0\leq t\leq T}
\left|A_{n,t}^{ij,\pm}(\vp_{n}(x))-A_t^{ij,\pm}(\vp(x))\right|>\veps\right\}\to
0, \ n\to\infty.
$$
\end{Lemma}

For the proof we make use of the following proposition
\begin{Proposition}\label{Proposition_Kulik}
{\it Let $X, Y$ be complete separable metric spaces, $(\Omega,
\mathcal{F}, {P})$ be a probability space. Let measurable mappings $\xi_n:\Omega\to
X,$ $h_n: X\to Y$, $n\geq 0$, be such that
\begin{enumerate}[1)]
\item $\xi_n\to\xi_0, \ n\to\infty,$ in probability; \item $h_n\to
h_0, \ n\to\infty,$ in measure $\nu$, where $\nu$ is a probability
measure on X;
\item for all $n\geq 1$ the distribution $P_{\xi_n}$
of $\xi_n$ is absolutely continuous w.r.t. the measure $\nu$; \item
the sequence of densities $\{\frac{dP_{\xi_n}}{d\nu}: \ n\geq1\}$
is uniformly integrable w.r.t. the measure $\nu$.
\end{enumerate}
Then $h_n(\xi_n)\to h_0(\xi_0),  \ n\to\infty,$ in probability.}
\end{Proposition}
The proof can be found, for example, in \cite{Bogachev07-2}, Corollary 9.9.11 or
\cite{Kulik+00}, Lemma 2.

\begin{proof}[Proof of Lemma \ref{Lemma_Converg_Functionals}]
Fix $T>0$ and $x\in\R$. Since $\varphi_t, \ \varphi_{n,t}, \ n\geq1,$
are measurable functions of a Wiener process, we may assume without loss of generality that  $\Omega=C([0,T], \R),$
$\mathcal{F}=\sigma\{w_t: 0\leq t\leq T\}$, $P=\P$ is the Wiener measure, and put
$\xi_n=\left(\vp_{n,t}(x)\right)_{0\leq t\leq T},$
$\xi_0=\left(\vp_t(x)\right)_{0\leq t\leq T}$, $\nu=\P_{w(x)}$ is the
distribution of the process $(w_t(x))_{0\leq t\leq T}$, $X=C([0,T],\R)$, $Y=C([0,T])$,  $h_n^{\pm}=A^{ij,\pm}_{n,t}(\cdot)$,  $h_0^{\pm}=A^{ij,\pm}_{t}(\cdot)$. Then $\{\xi_n:
n\geq 0\}$ is a sequence of random elements in the space $(\Omega,
\mathcal{F}, \P)$ taking values on $C([0,T],\R)$.
Lemma \ref{Lemma_Converg_Solutions} entails the convergence
$\xi_n\to\xi_0, \ n\to\infty,$ in probability $\P$ uniformly in
$t\in[0,T]$. This implies the first assertion of Proposition
\ref{Proposition_Kulik}.

According to Lemma \ref{Lemma_Converg_w_Functionals},
$A_{n,t}^{ij}\to A_t^{ij}$ as $n\to\infty,$ in probability measure $\P_{w(x)}$
uniformly in $t\in[0,T].$ This means that $h_n^{\pm}\to h^{\pm}, \ n\to\infty,$ as elements of  $C([0,T])$ in measure $\P_{w(x)}$.
 So the second assertion of Proposition \ref{Proposition_Kulik} is
justified. The absolute continuity of the distribution of
$(\vp_{n,t}(x))_{0\leq t\leq T}$ w.r.t. the measure $\P_{w(x)}$ follows from
Girsanov's theorem.  The density is defined by the formula
$$
\beta_n=\frac{d{\P}_{\vp_n(x)}}{d{\P_{w(x)}}}=\exp\left\{\int_0^T(a_n(w_s(x)),dw_s(x))-\frac{1}{2}\int_0^T
\|a_n(w_s(x))\|^2ds\right\}.
$$
As
$$
\mathds{E}\exp\left\{\frac{1}{2}\int_0^T
\|a_n(w_s(x))\|^2ds\right\}\leq \exp\left\{\frac{T}{2}\sup_{y\in\R}\|a(y)\|^2\right\}<\infty,
$$
where $\|\cdot\|$ is a norm in $\R$, we have that for each $p>1$,
$$
\mathds{E}\exp\left\{p\int_0^T(a_n(w_s(x)),dw_s(x))-\frac{p^2}{2}\int_0^T
\|a_n(w_s(x))\|^2ds\right\}=1
$$
(cf. \cite{Liptser74}, Theorem 6.1). The uniform integrability of
the family $\{\frac{dP_{\vp_n(x)}}{d{\P_{w(x)}}}: \ n\geq1\}$ follows from
the estimate
\begin{multline*}
\mathds{E}\exp \left\{ p \left( \int_0^T (a_n(w_s(x)),dw_s(x))-\frac{1}{2}\int_0^T \|a_n(w_s(x))\|^2ds\right)\right\}=\\
\mathds{E}\exp \left\{ p\int_0^T (a_n(w_s(x)),dw_s(x))-\frac{p^2}{2}\int_0^T \|a_n(w_s(x))\|^2ds\right\}\times\\
\exp\left\{ \frac{1}{2}(p^2-p)\int_0^T \|a_n(w_s(x))\|^2ds\right\}\leq\\
\exp\left\{(p^2-p)\|a_n\|_{\infty}^2 T\right\}\mathds{E}\exp\left\{p\int_0^T(a_n(w_s(x)),dw_s(x))-\frac{p^2}{2}\int_0^T\|a_n(w_s(x))\|^2ds\right\}=\\
\exp\left\{(p^2-p)\|a_n\|_{\infty}^2T\right\}\leq
\exp\left\{(p^2-p)\|a\|_{\infty}^2T\right\}
\end{multline*}
valid for $p>1$.
Thus all the assertions of Proposition \ref{Proposition_Kulik} are
fulfilled and we have
$$
\sup_{0\leq t\leq T}
\left|A_{n,t}^{ij,\pm}(\vp_{n}(x))-A_{t}^{ij,\pm}(\vp(x))\right|\to
0, \ n\to\infty,
$$
in probability $\P$. The Lemma is proved.
\end{proof}

\section{Convergence of the derivatives of solutions}\label{Section_Convergence_Derivatives} Recall that $ Y_t(x), Y_{n,t}(x),$ $t\geq 0, \ x\in\R,$ are the solutions of equations (\ref{eq_derivative_main}), (\ref{eq_derivative}), respectively.
In this section we  show the convergence of the sequence
$\{Y_{n,t}(x): n\geq1\}$ in probability uniformly in $t$. This
together with Lemma \ref{Lemma_Converg_Solutions} allow us to  prove  Theorem \ref{Theorem_main}.
\begin{Lemma}\label{Lemma_Converg_Derivatives}

\begin{enumerate}[1)]
\item For all $T\geq0$, $x\in\R$, $p>0$,
$$
\sup_{n\geq 1}\mathds{E}\sup_{0\leq t\leq T}\|Y_{n,t}(x)\|^p<\infty,
$$
\item
For all $T\geq0, \ x\in\R, p>0,$
$$
\E\sup_{0\leq t\leq T}\|Y_{n,t}(x)-Y_t(x)\|^p\to 0, \ n\to\infty, \
$$
\end{enumerate}
where
$$
\|Y\|=\max_{1\leq i,j \leq d}|Y^{ij}|.
$$
\end{Lemma}

For the proof we need the following two propositions.
The first one is a version  of the Gronwall-Bellman inequality
and can be obtained by a standard argument.
\begin{Proposition}\label{Prop_Gronwall_Lemma}
Let  $x(t)$ be a continuous function on
$[0,+\infty)$, $C(t)$ be a non-negative continuous function
on $[0,+\infty)$, $K(t)$ be a non-negative, non-decreasing function, and  $K(0)=0$. If for
all $0\leq t\leq T$,
$$
x(t)\leq C(t)+\left|\int_0^t x(s)dK(s)\right|,
$$
then
$$
x(T)\leq \left(\sup_{0\leq t\leq T}C(t)\right)\exp\{K(T)\}.
$$
\end{Proposition}
\begin{Proposition}\label{Proposition_moments_A}
For all $t\geq0$, $p>0,$ $1\leq i,j \leq d,$ there exists a constant $C$ such that
\begin{equation}\label{eq_exp_Moment_W-functional}
\sup_{x\in\R}\sup_{ n\geq1}\E\left(\exp\left\{pA_{n,t}^{ij,\pm}(\varphi_n(x))\right\}+\exp\left\{pA_{t}^{ij,\pm}(\varphi(x))\right\}\right)<C.
\end{equation}
\end{Proposition}
\begin{proof} The statement of the Proposition follows from Lemma \ref{Lemma_exponent_moment} and the inequalities (\ref{eq_gaussian_estimates}), which allow us to obtain the estimates uniform in $n\geq1$.
\end{proof}

\begin{proof}[Proof of Lemma \ref{Lemma_Converg_Derivatives}] For all $t>0$, define the variation of $A_{\cdot}^{ij}$ on $[0,t]$ by
$$
\Var A_t^{ij}(\varphi(x)):=A_t^{ij,+}(\varphi(x))+A_t^{ij,-}(\varphi(x)),
$$
and put
$$
\Var A_t(\varphi(x)):=\Sigma_{1\leq i,j \leq d}\Var A_t^{ij}(\varphi(x)).
$$

The variations of $A_{n,t}(\varphi_n(x)), \ n\geq 1,$ are defined similarly.

{\it The proof of 1)}.
We have
$$
\|Y_{n,t}(x)\|\leq 1+\left\|\int_0^t\left(dA_{n,s}(\varphi_n(x))\right)Y_{n,s}(x)\right\|\leq 1+\int_0^t\left\|Y_{n,s}(x)\right\|d\left(\Var A_{n,s}(\varphi_n(x))\right).
$$
Making use of the Gronwall-Bellman lemma
 we get
\begin{equation}\label{eq_estimate_derivative}
\|Y_{n,t}(x)\|\leq \exp\left\{\Var A_{n,t}(\varphi_n(x))\right\}\leq \exp\left\{\Var A_{n,T}(\varphi_n(x))\right\}.
\end{equation}
The statement 1) follows now from the estimate (\ref{eq_estimate_derivative}) and Proposition \ref{Proposition_moments_A}.

{\it The proof of 2)}.
  We have
\begin{multline*}
\|Y_{n,t}(x)-Y_t(x)\|\leq\left\|\int_0^t\left(dA_{n,s}(\varphi_n(x))-dA_s(\varphi(x))\right)Y_s(x)\right\|+\\
\left\|\int_0^tdA_{n,s}(\varphi_n(x))\left(Y_{n,s}(x)-Y_s(x)\right)\right\|\leq\\
\left\|\int_0^t\left(dA_{n,s}(\varphi_n(x))-dA_s(\varphi(x))\right)Y_s(x)\right\|+\int_0^t\left\|Y_{n,s}(x)-Y_s(x)\right\|d\left(\Var A_{n,s}(\varphi(x))\right).
\end{multline*}
By Proposition \ref{Prop_Gronwall_Lemma},
\begin{equation}\label{eq_derivative_exponent}
\|Y_{n,t}(x)-Y_t(x)\|\leq\sup_{0\leq u\leq t}\left\|\int_0^u\left(dA_{n,s}(\varphi_n(x))-dA_s(\varphi(x))\right)Y_s(x)\right\|\exp\left\{\Var A_{n,t}(\varphi_{n}(x))\right\}.
\end{equation}
To estimate the right-hand side of (\ref{eq_derivative_exponent}) we make use of the following Proposition.
\begin{Proposition}\label{Proposition_Monot_convergence} Let
$\{g_n : \ n\geq 1\}$ be a sequence of continuous monotonic
functions on $[0,T]$, and $f\in C([0,T]).$ Suppose that for each $t\in[0,T],$
$g_n(t)\to g(t),$ as $n\to\infty.$ Then
$$
\sup_{t\in[0,T]}\left|\int_0^t f(s)dg_n(s)-\int_0^t
f(s)dg(s)\right|\to 0, \ n\to\infty.
$$
\end{Proposition}

We get
\begin{multline}\label{eq_hhhh}
\sup_{0\leq u\leq
t}\left\|\int_0^u\left(dA_{n,s}(\varphi_n(x))-dA_s(\varphi(x))\right)Y_s(x)\right\|\exp\left\{\Var A_{n,t}(\varphi_n(x))\right\}\leq\\
\sup_{0\leq u\leq
t}\left\|\int_0^u\left(dA_{n,s}^{+}(\varphi_n(x))-dA_s^{+}(\varphi(x))\right)Y_s(x)\right\|
\exp\left\{\Var A_{n,t}(\varphi_n(x))\right\}+\\
\sup_{0\leq u\leq t}
\left\|\int_0^u\left(dA_{n,s}^{-}(\varphi_n(x))-dA_s^{-}(\varphi(x))\right)Y_s(x)\right\|
\exp\left\{\Var A_{n,t}(\varphi_n(x))\right\}.
\end{multline}
Consider the first summand in the right-hand side  of (\ref{eq_hhhh}). Put $g_n(s)=A_{n,s}^{+}(\varphi_n(x))$, $g(s)=A_s^{+}(\varphi(x)),$  and $f(s)=Y_s(x).$
Then
Lemma \ref{Lemma_Converg_Functionals}, Proposition \ref{Proposition_moments_A}, and Proposition \ref{Proposition_Monot_convergence}
provide that
$$
\sup_{0\leq u\leq
t}\left\|\int_0^u\left(dA_{n,s}^{+}(\varphi_n(x))-dA_s^{+}(\varphi(x))\right)Y_s(x)\right\|
\exp\left\{\Var A_{n,t}(\varphi_n(x))\right\}\to 0 \ \mbox {as} \ n\to\infty,
$$
in probability.
Similarly it is proved that the second summand in the right-hand side of (\ref{eq_hhhh}) tends to $0$ as $n\to\infty$.

This and statement 1) entail statement 2) of the Lemma.
\end{proof}

\section{The proof of Theorem \ref{Theorem_main}}\label{Section_Proof}
\begin{proof} Define approximating equations by (\ref{eq_main_n}),
where $a_n, \ n\geq 1,$ are determined by (\ref{eq_a_n}).  From
Lemma \ref{Lemma_Converg_Solutions} and the dominated convergence
theorem we get the relation
$$
\mathds{E}\sup_{t\in[0,T]}\int_U|\varphi_{n,t}^i(x)-\varphi_t^i(x)|^pdx\to 0 , \
n\to \infty,
$$
valid for any bounded domain $U\subset\R$, $T>0$, $p\geq 1,$ and
$1\leq i \leq d.$ So for each $1\leq i \leq d,$ there exists a
subsequence $\{n_k^i: \ k\geq 1\}$ such that
$$
\sup_{t\in[0,T]}\int_U|\varphi_{n_k^i,t}^i(x)-\varphi_{t}^i(x)|^pdx\to 0 \
\mbox{a.s. as} \ k\to \infty.
$$
Without loss of generality we can suppose that
\begin{equation}\label{eq_convergence_solutions}
\sup_{t\in[0,T]}\int_U|\varphi_{n,t}^i(x)-\varphi_{t}^i(x)|^pdx\to 0 \ \mbox{a.s.
as}\ n\to \infty.
\end{equation}

Arguing similarly and taking into account Lemma
\ref{Lemma_Converg_Derivatives} we arrive at the relation
\begin{equation}\label{eq_convergence_derivatives}
\sup_{t\in[0,T]}\int_U|Y_{n,t}^{ij}(x)-Y_{t}^{ij}(x)|^pdx\to 0, \ n\to\infty, \
\mbox{almost surely},
\end{equation}
which is fulfilled for all $1\leq i,j \leq d,$ $p\geq 0.$

Since the Sobolev space is a Banach space,  the relations
(\ref{eq_convergence_solutions}),
(\ref{eq_convergence_derivatives}) mean that $Y_t(x)$ is the
 matrix
of derivatives of the solution to (\ref{eq_main}).

Let us verify \eqref{eq_derivative_main2}. We have for all $x,h\in\R, \alpha\in\mathbb{R}, $
$$
\vf_{n,t}(x+\alpha h)=\vf_{n,t}(x)+\int_0^\alpha Y_{n,t}(x+uh)du.
$$
It follows from Lemmas \ref{Lemma_Converg_Solutions} and \ref{Lemma_Converg_Derivatives} that
\begin{equation}\label{eq_NL}
\vf_{t}(x+\alpha h)=\vf_{t}(x)+\int_0^\alpha Y_{t}(x+uh)du.
\end{equation}
The following lemma implies the relation
\begin{equation}\label{eq_30.1}
\forall y_0\in\R:\ Y_t(y)\to Y_t(y_0), \ y\to y_0,
\end{equation}
in probability and hence  in all $L_p$. This completes the proof of the Theorem, as \eqref{eq_NL}
and \eqref{eq_30.1}
implies \eqref{eq_derivative_main2}.
\end{proof}

\begin{Lemma} Let $\nu$ be a measure of Kato's class. Then for any $t\geq0$, $x_0\in\R$, $\varepsilon>0$,
\begin{equation}\label{eq_continuity_w_functionals}
\P\left\{\left|A_t^{{\nu}}(\vf(x))- A_t^{{\nu}}(\vf(x_0))\right|>\varepsilon\right\}\to 0 \ \mbox{as} \ x\to x_0.
\end{equation}
\end{Lemma}
\begin{proof} For $e\in\R$, denote by $\nu_{e}$ the shift of the measure $\nu$ by the vector $e$, i.e. for each $A\subset\R$,
$$
\nu_e(A)=\nu({x:x-e\in A}).
$$
Then
$$
A_t^{\nu}(\varphi(x))=A_t^{\nu_{x-x_0}}(\varphi_{\cdot}(x)-x+x_0).
$$
Note that for fixed $x$ and $x_0$ the process $(\xi_t)_{t\geq 0}:=(\varphi_t(x)-x+x_0)_{t\geq0}$ can be considered as a Markov process starting from $x_0$, and its distribution is equivalent to the distribution $\P_{w(x_0)}$  of the Wiener process starting from $x_0$. Indeed,

$$
\xi_t=x_0+\int_0^t \tilde a(\xi_s)ds+w(t),
$$
where $\tilde a(y)=a(y+x-x_0).$
Similarly to the proof of Lemma 5 it can be checked that the family of the Radon-Nikodym densities $\left\{\frac{d\P_{\varphi_{\cdot}(x)-x+x_0}}{d\P_{w(x_0)}}, \ x\in\R\right\}$ are uniformly integrable with respect to $P_{w(x_0)}$.
By Proposition \ref{Lemma_Converg_Solutions} and  Lemma \ref{Proposition_Kulik} to prove (\ref{eq_continuity_w_functionals}) it suffices to verify that
\begin{equation}\label{eq_convergence_w_functionals_2}
A_t^{\nu_{x-x_0}}({w(x_0)})\to A_t^{\nu}({w(x_0)}), \ x\to x_0, \ \mbox{in probability} \ \P.
\end{equation}

By $\nu^{(R)}(dy)=\1_{|y|\leq R}\nu(dy)$ denote the restriction of the measure $\nu$ to the ball $\{y:\ |y|\leq R\}.$ Put $f_t^w(y)=\E A_t^{\nu}({w(y)})$, $f_{R,t}^w(y)=\E A_t^{\nu^{(R)}}({w(y)})$. Then
$$
\E A_t^{\nu_{x-x_0}}(w(y))=f_t^w(y+x-x_0),\ \  \E A_t^{\nu^{(R)}_{x-x_0}}(w(y))=f_{R,t}^w(y+x-x_0).
$$
It is easy to see that the function $(s,y)\to f_{R,t}^w(y)$ is uniformly continuous in $(s,y)\in[0,t]\times \R.$
So by Theorem \ref{Theorem_Convergence_characteristics} we have the convergence in probability
\begin{equation}\label{eq_shodimost}
A_t^{\nu^{(R)}_{x-x_0}}(w(y))\to A_t^{\nu^{(R)}}(w(y)), \ x\to x_0,
\end{equation}
for any $y\in\R.$

It follows from \cite{Dynkin63}, Theorem 8.4  that for any $R>0$ and $y\in\R$ we have the equality $A_t^{\nu^{(R)}}({w(y)})=A_t^{\nu}({w(y)})$ a.s. on the set
$\{\sup_{s\in[0,t]}|y+w_s|<R\}.$
This together with \eqref{eq_shodimost} entails (\ref{eq_convergence_w_functionals_2}).
\end{proof}

\section{Appendix: The proof of Lemma \ref{Lemma_equiv_w_functionals}}
Note that $B_t^h(\varphi(x))=\int_0^t\frac{f_h^w(\varphi_s(x))}{h}ds$ is a W-functional. Let us estimate its characteristic.
$$
\mathds{E}B_t^h(\varphi(x))=\mathds{E}\int_0^t\frac{f_h^w(\varphi_s(x))}{h}ds=\frac{1}{h}\int_0^hdu\int_{\R}\left(\int_0^tds\int_{\R}p_u^w(z,y)p_s^{\varphi}(x,z)dz\right)\nu(dy).
$$
From the estimates (\ref{eq_gaussian_estimates}) we obtain (see also the proof of Theorem 6.6. in \cite{Dynkin63})
\begin{multline*}
\mathds{E}B_t^h(\varphi(x))\leq\\
\frac{1}{h}\int_0^hdu\int_{\R}\left(\int_0^{t}ds\int_{\R}\frac{K}{u^{d/2}}\exp\left\{-\frac{k\|y-z\|^2}{u}\right\}\frac{K}{s^{d/2}}
\exp\left\{-\frac{k\|z-x\|^2}{s}\right\}dz\right)\nu(dy)=\\
\widetilde{K}\frac{1}{h}\int_0^hdu\int_{\R}\left(\int_0^{t}\frac{1}{(2\pi(u+s))^{d/2}}\exp\left\{-\frac{k\|y-x\|^2}{u+s}\right\}ds\right)\nu(dy)=\\
\widetilde{K}\frac{1}{h}\int_0^hdu\int_{\R}\left(\int_u^{t+u}\frac{1}{(2\pi s)^{d/2}}\exp\left\{-\frac{k\|y-x\|^2}{s}\right\}ds\right)\nu(dy)=\\
\widehat{K}\frac{1}{h}\int_0^hdu\int_{\R}\left(\int_{{u}/{2k}}^{(t+u)/2k}\frac{1}{(2\pi s)^{d/2}}\exp\left\{-\frac{\|y-x\|^2}{2s}\right\}ds\right)\nu(dy)=\\
\widehat{K}\frac{1}{h}\int_0^h\left(f^w_{(t+u)/{2k}}(x)-f_{{u}/{2k}}^w(x)\right)du.
\end{multline*}
where $\widetilde{K}=K^2\pi^2(2/k)^{d/2},$ $\widehat K=2K^2k^{1-d}\pi^d$.
Taking into account  (\ref{eq_additive}), we get
$$
f_{(t+u)/{2k}}^w(x)-f_{{u}/{2k}}^w(x)=T^w_{u/2k}f^w_{t/{2k}}(x)\leq\|f^w_{t/2k}\|_{\infty}.
$$
By Proposition \ref{Prop_Gikhman_Skorokhod},
\begin{equation}\label{eq_similar}
\sup_{x\in\R}\mathds{E}_x\left(B_t^h(\varphi)\right)^2\leq 2\widehat{K}^2\left(\|f^w_{t/2k}\|_{\infty}\right)^2.
\end{equation}
Therefore, the second moment of $B_t^h(\varphi)$ is  bounded uniformly in $h$. This implies  the uniform integrability and, consequently the convergence in $L_1$ holds in (\ref{eq_Girsanov_theorem}).
Then the characteristic of the functional $A_t^{\nu,w}(\varphi(x))$ is equal to
$$
\widetilde{f}_t(x)=\lim_{h\downarrow 0}\mathds{E}\int_0^t \frac{f_h^w(\varphi_s(x))}{h}ds.
$$
If we show that
\begin{equation}\label{eq_f_tilde}
\widetilde{f}_t(x)=\int_{\R}k_t^{\varphi}(x,y)\nu(dy),
\end{equation}
then the statement of the Lemma follows from Proposition \ref{Proposition_uniquely_defined}.
We have, for each $0<\delta<t$,
\begin{multline*}\label{eq_finiteness_characteristic }
\left|\mathds{E}\int_0^t\frac{f_h^w(\varphi_s(x))}{h}ds-\int_{\R}k_t^{\varphi}(x,y)\nu(dy)\right|\leq\\
\mathds{E}\int_0^{\delta}\frac{f_h^w(\varphi_s(x))}{h}ds+\int_0^{\delta}\left(\int_{\R}p_s^{\varphi}(x,y)\nu(dy)\right)ds+\\
\left|\mathds{E}\int_{\delta}^t\frac{f_h^w(\varphi_s(x))}{h}ds-\int_{\delta}^t\left(\int_{\R}p_s^{\varphi}(x,y)\nu(dy)\right)ds\right|=I+II+III.
\end{multline*}
Consider $I$. Arguing as in the proof of (\ref{eq_similar}) we arrive at the inequality
$$
I\leq \|f^w_{\delta/2k}\|_{\infty}.
$$
Making use of (\ref{eq_gaussian_estimates}) and changing the variables we get
\begin{multline*}
II\leq 2K\pi^{d/2}(k)^{1-d/2}\int_0^{\delta/2k}ds\int_{\R}\frac{1}{(2\pi s)^{d/2}}\exp\left\{-\frac{\|y-x\|^2}{2s}\right\}\nu(dy)\\
\leq 2K\pi^{d/2}(k)^{1-d/2}\|f^w_{\delta/2k}\|_{\infty}.
\end{multline*}
For each $\varepsilon>0$, the condition (\ref{Cond_A_prime}) allows us to choose $\delta$ so small that
\begin{equation}\label{eq_I_II}
I<\varepsilon/3, \ II<\varepsilon/3.
\end{equation}
Further,
$$
III=\left|\int_{\delta}^t ds\int_{\R}\nu(dy)\int_{\R}(p_s^{\varphi}(x,z)-p_s^{\varphi}(x,y))\left(\frac{1}{h}
\int_0^hp_u^w(z,y)du\right)dz\right|.
$$
The measure $\frac{1}{h}
\left(\int_0^h p_u^w(z,y)du\right)dz$ converges weakly to the $\delta$-measure at the point $y$. The function $p^{\varphi}_s(x,y)$ is equicontinuous in $y$ for $s\in[\delta,t], \ x\in\R$. So
$$
\int_{\R}(p_s^{\varphi}(x,z)-p_s^{\varphi}(x,y))\left(\frac{1}{h}
\int_0^hp_u^w(z,y)du\right)dz\to 0, \ h\downarrow 0,
$$
uniformly in $x$ and $s$.
Besides, from (\ref{eq_gaussian_estimates})
\begin{multline*}
\left|\int_{\R}(p_s^{\varphi}(x,z)-p_s^{\varphi}(x,y))\left(\frac{1}{h}
\int_0^hp_u^w(z,y)du\right)dz\right|\leq\\
\int_{\R}\frac{K}{s^{d/2}}\left(\exp\left\{-\frac{k\|x-z\|^2}{s}\right\}+
\exp\left\{-\frac{k\|x-y\|^2}{s}\right\}\right)\left(\frac{1}{h}
\int_0^hp_u^w(z,y)du\right)dz\leq \frac{2K}{s^{d/2}}.
\end{multline*}
By the dominated convergence
theorem,
\begin{equation}\label{eq_III}
III\to 0 \ \mbox{as} \ \ h\downarrow 0.
\end{equation}
Now the equality (\ref{eq_f_tilde}) follows from (\ref{eq_I_II}) and (\ref{eq_III}). The Lemma is proved.

\section*{Acknowledgements}
 The authors thank Prof. Le Jan and Prof. Kulik for fruitful discussions. We are appreciate Prof. Portenko for his useful remarks to the manuscript. We also are grateful to the anonymous referee for his thorough reading and valuable comments which helped to improve essentially the exposition.


\begin{thebibliography}{10}

\bibitem{Aizenman+82}
M.~Aizenman and B.~Simon.
\newblock Brownian motion and Harnack inequality for {S}chr\"{o}dinger operators.
\newblock {\em Communications on Pure and Applied Mathematics}, 35(2):209--273,
  1982.

\bibitem{Aronson67}
D.~G. Aronson.
\newblock Bounds for the fundamental solution of a parabolic equation.
\newblock {\em Bull. Amer. Math. Soc.}, 73:890--896, 1967.

\bibitem{Aryasova+12}
O.~V. Aryasova and A.~Yu. Pilipenko.
\newblock On properties of a flow generated by an {SDE} with discontinuous
  drift.
\newblock {\em Electron. J. Probab.}, 17:no. 106, 1--20, 2012.

\bibitem{Attanasio10}
S.~Attanasio.
\newblock Stochastic flows of diffeomorphisms for one-dimensional {SDE} with
  discontinuous drift.
\newblock {\em Electron. Commun. Probab.}, 15:no. 20, 213--226, 2010.

\bibitem{Bass+03}
R.~F. Bass and Z.-Q. Chen.
\newblock Brownian motion with singular drift.
\newblock {\em The Annals of Probability}, 31(2):791--817, 04 2003.

\bibitem{Blumental+07}
R.~M. Blumenthal and R.~K. Getoor.
\newblock {\em {Markov Processes and Potential Theory. Reprint of the 1968
  ed.}}
\newblock {Mineola, NY: Dover Publications. vi, 313~p.}, 2007.

\bibitem{Bogachev07-2}
V.~I. Bogachev.
\newblock {\em Measure Theory}, volume~2.
\newblock Springer, Berlin, 2007.

\bibitem{Dynkin63}
E.~B. Dynkin.
\newblock {\em Markov Processes}.
\newblock Fizmatlit, Moscow, 1963.
\newblock [Translated from the Russian to the English by J. Fabius, V.
  Greenberg, A. Maitra, and G. Majone. Academic Press, New York; Springer,
  Berlin, 1965. vol. 1, xii + 365 pp.; vol. 2, viii + 274 pp.].

\bibitem{Federer69}
{Federer H.}
\newblock {\em Geometric Measure Theory}, volume 153 of {\em {Die Grundlehren
  der mathematischen Wissenschaften}}.
\newblock {New York}, {Springer-Verlag New York Inc.} edition, 1969.

\bibitem{Fedrizzi+13b}
E.~Fedrizzi and F.~Flandoli.
\newblock H{\"{o}}lder flow and differentiability for {SDE}s with nonregular
  drift.
\newblock {\em Stochastic Analysis and Applications}, 31(4):708--736, 2013.

\bibitem{Fedrizzi+13a}
E.~Fedrizzi and F.~Flandoli.
\newblock Noise prevents singularities in linear transport equations.
\newblock {\em Journal of Functional Analysis}, 264(6):1329 -- 1354, 2013.

\bibitem{Flandoli+10}
F.~Flandoli, M.~Gubinelli, and E.~Priola.
\newblock Flow of diffeomorphisms for {SDE}s with unbounded {H}{\"{o}}lder
  continuous drift.
\newblock {\em Bulletin des Sciences Mathematiques}, 134(4):405 -- 422, 2010.

\bibitem{Gikhman+04_II}
I.~I. Gikhman and A.~V. Skorokhod.
\newblock {\em The Theory of Stochastic Processes. II}.
\newblock Nauka, Moscow, 1973.
\newblock [Translated from the Russian by S. Kotz. Corrected printing of the
  first edition. Berlin: Springer, 2004. viii, 441~p.].

\bibitem{Khasminskii59}
R.~Z. Khas�minskii.
\newblock On positive solutions of the equation $\mathfrak{A}u + vu = 0$.
\newblock {\em Theory of Probability and Its Applications}, 4(3):309--318,
  1959.

\bibitem{Kulik+00}
A.~M. Kulik and A.~Yu. Pilipenko.
\newblock Nonlinear transformations of smooth measures on infinite-dimensional
  spaces.
\newblock {\em Ukrainian Mathematical Journal}, 52:1403--1431, 2000.
\newblock 10.1023/A:1010380119199.

\bibitem{Kuwae+07}
K.~Kuwae and M.~Takahashi.
\newblock Kato class measures of symmetric {M}arkov processes under heat kernel
  estimates.
\newblock {\em Journal of Functional Analysis}, 250(1):86 -- 113, 2007.

\bibitem{Liptser74}
R.~S. Liptser and A.~N. Shiryayev.
\newblock {\em Statistics of Random Processes}.
\newblock Nauka, Moscow, 1974.
\newblock [Translated from the Russian to the English by A. B. Aries.
  Springer-Verlag, New York, 1977].

\bibitem{Luo11}
D.~Luo.
\newblock Absolute continuity under flows generated by {SDE} with measurable
  drift coefficients.
\newblock {\em Stochastic Processes and their Applications}, 121(10):2393 --
  2415, 2011.

\bibitem{MeyerBrandis+13}
T.~Meyer-Brandis and F.~Proske.
\newblock {Construction of strong solutions of SDE's via Malliavin calculus}.
\newblock {\em Journal of Functional Analysis}, 258(11):3922 -- 3953, 2010.

\bibitem{Mohammed+12}
S.E.A. Mohammed, T.~Nilssen, and F.~Proske.
\newblock Sobolev differentiable stochastic flows of {SDE}'s with measurable
  drift and applications.
\newblock  {\em Annals of Probability} (to appear).

\bibitem{Revuz+99}
D.~Revuz and M.~Yor.
\newblock {\em Continuous Martingales and Brownian Motion}.
\newblock Springer-Verlag, Berlin, 1999.

\bibitem{Sznitman98}
A.-S. Sznitman.
\newblock {\em Brownian motion, obstacles, and random media}.
\newblock Springer monographs in mathematics. Springer, 1 edition, 1998.

\bibitem{Veretennikov81}
A.~Y. Veretennikov.
\newblock On strong solutions and explicit formulas for solutions of stochastic
  integral equations.
\newblock {\em Math. USSR Sborn}, 39(3):387--403, 1981.

\bibitem{Vladimirov67}
V.~S. Vladimirov.
\newblock {\em The Equation of Mathematical Phisics}.
\newblock Nauka, Moscow, 1967.
\newblock [Translated from the Russian to the English by A. Littlewood. Marcel
  Dekker, INC., New York, 1971.].

\end{thebibliography}

\end{document}